\documentclass[preprint, 11pt]
{elsarticle}
\usepackage{microtype}
\usepackage[english]{babel}
\usepackage{amsfonts,amsmath,amsthm,mathtools}
\usepackage{color}
\usepackage{amssymb}
\usepackage{pstricks,pst-node,pst-text,pst-3d}
\usepackage{tikz}

\usetikzlibrary{arrows}

\usepackage{environ}
\makeatletter
\providecommand{\env@tikzpicture@save@env}{}
\providecommand{\env@tikzpicture@process}{}

\makeatother

\numberwithin{figure}{section}
\numberwithin{equation}{section}

\theoremstyle{plain}
\newtheorem{theorem}{Theorem}[section]
\newtheorem{corollary}{Corollary}[section]
\newtheorem{lemma}{Lemma}[section]
\newtheorem{proposition}{Proposition}[section]

\newtheorem*{convention}{Convention}
\newtheorem*{BRT}{Birkhoff's representation theorem}

\theoremstyle{definition}
\newtheorem{definition}{Definition}[section]
\newtheorem{example}{Example}[section]
\newtheorem{remark}{Remark}[section]

\newcommand{\dR}{\mathbb{R}}

\newcommand{\calC}{\mathcal{C}}
\newcommand{\calD}{\mathcal{D}}
\newcommand{\calF}{\mathcal{F}}

\newcommand{\calJ}{\mathcal{J}}
\newcommand{\calL}{\mathcal{L}}

\newcommand{\calS}{\mathcal{S}}
\newcommand{\calT}{\mathcal{T}}
\newcommand{\calU}{\mathcal{U}}
\newcommand{\calV}{\mathcal{V}}

\newcommand{\bx}{\mathbf{x}}
\newcommand{\by}{\mathbf{y}}

\renewcommand{\phi}{\varphi}
\renewcommand{\epsilon}{\varepsilon}
\providecommand{\abs}[1]{\lvert#1\rvert}

\DeclareMathOperator{\relint}{relint}
\DeclareMathOperator{\verts}{ext}
\DeclareMathOperator{\da}{\downarrow}
\DeclareMathOperator{\DA}{\Downarrow}

\journal{Discrete Applied Mathematics}

\begin{document}

\title{The cone of supermodular games on finite distributive lattices}
\author[PSE]{Michel Grabisch}
\ead{michel.grabisch@univ-paris1.fr}
\author[UTIA]{Tom\'{a}\v{s} Kroupa\corref{cor}}
\ead{kroupa@utia.cas.cz}
\address[PSE]{Paris School of Economics, University of Paris I,\\ Paris, France}
\address[UTIA]{The Czech Academy of Sciences, Institute of Information Theory and Automation,\\ Prague, Czech Republic}
\cortext[cor]{Corresponding author}
\begin{abstract}
In this article we study supermodular functions on finite distributive
lattices. Relaxing the assumption that the domain is a powerset of a finite set,
we focus on geometrical properties of the polyhedral cone of such
functions. Specifically, we generalize the criterion for extremality and study
the face lattice of the  supermodular cone. An explicit description of facets by the corresponding tight linear inequalities is provided.

\end{abstract}

\begin{keyword}
supermodular/submodular function \sep core \sep coalitional game \sep polyhedral cone
\end{keyword}

\maketitle

\section{Introduction}
Supermodular functions, and their duals, submodular functions, play a~central
r\^ole in many fields of discrete mathematics, most notably combinatorial
optimization (rank function of polymatroids: see, e.g., the monograph of Fujishige \cite{Fujishige05}), game
theory (characteristic function of transferable utility games: see, e.g., Peleg
and Sudh\"olter \cite{PelegSudholter07}), decision theory (capacity, Choquet
expected utility \cite{sch89}), lattice theory, etc.

Up to duality, all above examples fall into the category of {\it supermodular
  games}, that is, supermodular functions vanishing at the empty set. They form
a polyhedral cone, whose facets have been found by Kuipers et
al.~\cite{KuipersVermeulenVoorneveld10}. In his 1971 seminal paper, Shapley
\cite{Shapley71} gave the 37 extreme supermodular games for $n=4$ players, and
noted that for larger values of $n$, little can be said. Later, Rosenm\"uller
and Weidner \cite{rowe74} found all extreme supermodular functions by representing each such function as a maximum over shifted additive games. Recently, Studen\'y
and Kroupa \cite{StudenyKroupa:CoreExtreme} revisited the problem and provided
another characterization of extremality, in a sense dual to the result of
Rosenm\"uller and Weidner, but easier to use.

The aim of this paper is to (re)establish in a more general framework and in a
 simpler way the above results (together with new ones) describing the cone
of supermodular games, taking advantage of classical results on polyhedra. We
consider games defined on a finite distributive lattice $\calL$, generated by a partial
order $\preceq$ on the set of players $N$.  The poset induces some relation
between the players, which can be interpreted in various ways: precedence
constraints (Faigle and Kern \cite{fake92}), hierarchy (Grabisch and Xie
\cite{grxi11}), or permission structure (van den Brink and Gilles \cite{brgi96}). Feasible coalitions of players, i.e., those for which the game is defined,
are down-sets on $(N,\preceq)$, and they form  a
distributive lattice $\calL$.  By Birkhoff's theorem, every finite distributive lattice is of this form. The standard case $\calL=2^N$ is recovered when
the poset $(N,\preceq)$ is flat, i.e., when all players are incomparable (no
order relation between the players).

A large amount of research has been done concerning games on distributive
lattices, as well as on other ordered structures (see a survey in
\cite{gra13}). Most of them are related to the solution concepts such as Shapley value or the core. However, up to our knowledge, there is no systematic study on the
geometric properties of the cone of supermodular games defined on distributive
lattices. Note that it is very natural to take a distributive lattice as a domain of a supermodular function since supermodular inequalities involve only the lattice joins and meets. The present paper addresses precisely this point. In the same time we generalize and prove results about extreme rays and facets in a more concise way.

Section~\ref{sec:birkhoff} collects background on 
distributive lattices. Coalitional games are introduced in Section~\ref{sec:gamdislat}, in particular 0-normalized and supermodular games. Section \ref{sec:consup} contains basic facts about 0-normalized supermodular games and the cone thereof. The extreme rays  are characterized in
Section~\ref{sec:mainresult}. Basically, a supermodular game generates an extreme ray if and only if a certain system of linear
equalities has for a solution those vectors which are proportional to the marginal
vectors of the game. Section~\ref{sec:faces} describes the facial structure of
the cone by a certain collection of finite lattices, namely the tight sets
associated with compatible permutations of the poset $(N,\preceq)$. The facets
of the cone of supermodular games are characterized in Section~\ref{sec:facets}.

\section{Finite distributive lattices}\label{sec:birkhoff}
In this section we introduce basic notions and results about Birkhoff duality between finite distributive lattices and finite posets. The reader is referred to \cite[Chapter 3]{Stanley2012} for all the unexplained notions concerning lattices and partially ordered sets (posets).

Let $\calL$ be a finite distributive lattice whose join and meet are denoted by
$\vee$ and $\wedge$, respectively. A partial order $\leqslant$ on $\calL$ is
defined by $a\leqslant b$ if $a\vee b=b$, for all $a,b\in \calL$. Since $\calL$
is finite there exists a top element $\top$ and a bottom element $\bot$ in $\calL$. We always assume that $\calL$ is non-trivial in sense that $\top\neq \bot$. An element $a\in \calL$  is called \emph{join-irreducible} if $a\neq\bot$ and the identity $a=b\vee c$ holding for some $b,c\in \calL$ implies $a=b$ or $a=c$. In particular, $a\in\calL$ with $a\neq \bot$ is an \emph{atom} if the condition $b \leqslant a$ for all $b\in\calL$ implies $b=\bot$ or $b=a$. The join-irreducible elements of a Boolean lattice are precisely its atoms. For any $a,b\in \calL$ such that $a\leqslant b$, we define an \emph{order interval} $$[a,b]\coloneqq\{c\in \calL\mid a\leqslant c\leqslant b\}.$$ An element $a\in \calL$ is join-irreducible if, and only if, there is a unique $a^- \in \calL$ such that $a^-\leqslant a$, $a^-\neq a$, and $[a^-,a]=\{a^-,a\}$. The set of all join-irreducible elements of $\calL$ is denoted by $\calJ(\calL)$ and it is always endowed with the partial order $\leqslant$ of $\calL$ restricted to $\calJ(\calL)$. Thus, $(\calJ(\calL),\leqslant)$ becomes a~nonempty finite poset.

 Let $N\neq\emptyset$ be a finite set and $\preceq$ be a partial order on $N$. A \emph{down-set} in $(N,\preceq)$ is a subset $A\subseteq N$ such that if $i\in A$ and $j\preceq i$ for $j\in N$, then $j\in A$. For any $i \in N$, we denote
\[
\da i \coloneqq  \{j\in N\mid j\preceq i \} \qquad \text{and}\qquad
  \DA i  \coloneqq \da i \setminus \{i\}.
\]
Both $\da i$ and $\DA i$ are down-sets in $(N,\preceq)$. A down-set $A$ is called \emph{principal} if there exists some $i\in N$ such that $A=\da i$. By $\calD(N,\preceq)$ we denote the set of all down-sets in $(N,\preceq)$. It is easy to see that $\calD(N,\preceq)$ is closed under the set-theoretic union $\cup$ and intersection $\cap$. Thus, $\calD(N,\preceq)$ is a finite distributive lattice whose order is the inclusion $\subseteq$ between sets, and whose top and bottom element is $N$ and  $\emptyset$, respectively. The lattice $\calD(N,\preceq)$ is the most general example of a finite distributive lattice by the following classical result.

  \begin{BRT}
Let $\calL$ be a finite distributive lattice. Then the mapping $$F_{\calL}\colon \calL\to \calD(\calJ(\calL),\leqslant)$$ defined by $F_{\calL}(a)\coloneqq \{b \in \calJ(\calL) \mid b\leqslant a\}$ is a lattice isomorphism.
  \end{BRT}
  \noindent
The converse part of duality explains what are  join-irreducible elements in the lattice of down-sets $\calD(N,\preceq)$.
 \begin{proposition}\label{pro:embedding}
 	Let $(N,\preceq)$ be a finite poset. Then the mapping $$\da\colon (N,\preceq)\to \calJ(\calD(N,\preceq))$$ sending every $i\in N$ to the principal down-set $\da i$
 	is an order isomorphism.
 \end{proposition} 
 
 \begin{example}\label{example}
 Let $N\coloneqq \{1,2,3,4\}$ be equipped with the partial order $\preceq$ captured by the Hasse diagram in Figure \ref{fig:ex} on the left. On the right-hand side we depict the lattice of down-sets $\calD(N,\preceq)$. There are four join-irreducible elements in $\calD(N,\preceq)$, namely $\{2\}$, $\{3\}$, $\{4\}$, and $\{1,2,3\}$.

 \begin{figure}
  \begin{center}
 \begin{tikzpicture}[scale=.6]
 \draw[fill] (0,5) circle (3pt) node[above] {$1$} -- (-1.5,3) circle (3pt) node[below] {$2$};
  \draw[fill] (0,5) -- (1.5,3) circle (3pt) node[below] {$3$};
\draw[fill] (3,4) circle (3pt) node[below] {$4$};

\node (empty) at (10,0) {$\emptyset$};
\node (two) at (8,2) {$\{2\}$};
\node (three) at (10,2) {$\{3\}$};
\node (four) at (12,2) {$\{4\}$};
\node (tt) at (8,4) {$\{2,3\}$};
\node (tf) at (10,4) {$\{2,4\}$};
\node (thf) at (12,4) {$\{3,4\}$};
\node (ott) at (8,6) {$\{1,2,3\}$};
\node (ttf) at (11.7,6) {$\{2,3,4\}$};
\node (en) at (10,8) {$N$};
\draw (en) -- (ott); \draw (en) -- (ttf);
\draw (ott) -- (tt); \draw (ttf) -- (tt); \draw (ttf) -- (tf); \draw (ttf) -- (thf);
\draw (tt) -- (two); \draw (tt) -- (three); \draw (tf) -- (two); \draw (tf) -- (four);
\draw (thf) -- (three); \draw (thf) -- (four);
\draw (two) -- (empty); \draw (three) -- (empty); \draw (four) -- (empty);
\end{tikzpicture}
\end{center}
\label{fig:ex}
\caption{The poset $(N,\preceq)$ with the corresponding lattice of down-sets $\calD(N,\preceq)$}
\end{figure}
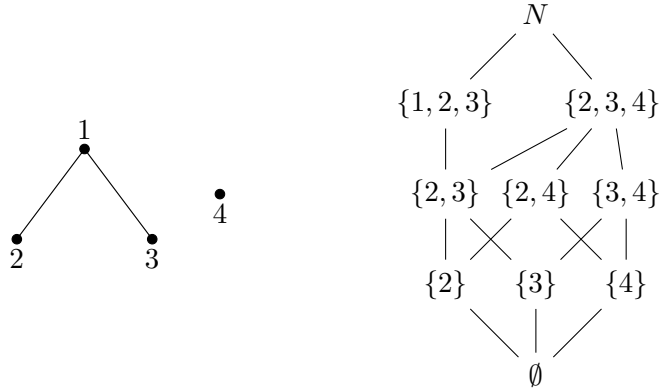

 \end{example}

\section{Coalitional games on finite distributive lattices}\label{sec:gamdislat}
We use the standard terminology of cooperative game theory; see \cite{PelegSudholter07}. The \emph{player set} is defined to be $N\coloneqq \{1,\dots,n\}$, for some integer $n\geq 1$. Any subset of $N$ is called a \emph{coalition}. We allow for a situation in which players $i,j\in N$ are compared using a partial order $\preceq$ on $N$. Hence,  $(N,\preceq)$ is assumed to be a finite poset. Birkhoff duality (see Section \ref{sec:birkhoff}) entails that the partial order $\preceq$ on $N$ restricts the formation of coalitions $A\subseteq N$, provided that the coalition structure is modeled by the lattice of down-sets in $(N,\preceq)$.

\begin{convention}
	Throughout the paper we will always assume that the set of all possible coalitions in $(N,\preceq)$ is the lattice of down-sets $\calD(N,\preceq)$. We use the abbreviations
\[
		\calL \coloneqq \calD(N,\preceq) \qquad \text{and} \qquad
		\calJ \coloneqq\calJ(\calD(N,\preceq)).
\]
\end{convention}
\noindent
 From now on, all possible coalitions are assumed to be precisely the sets belonging to a fixed lattice $\calL$ and $\calJ$ denotes its subset of all join-irreducible elements. Coalitional games are modeled as real functions $v$ on the set $\calL$ of feasible coalitions $A$, where the real value $v(A)$ indicate the amount of utility resulting from the joint cooperation of players in the coalition $A$.
 
 \begin{definition}
   A function $v\colon \calL\to \dR$ satisfying $v(\emptyset)=0$ is a \emph{(coalitional) game}. A game $v$ is called
   \begin{itemize}
   	\item \emph{supermodular} if 
   	$v(A\cup B) + v(A\cap B) \geq v(A) + v(B)$,
	\item \emph{modular} if $v(A\cup B) + v(A\cap B) = v(A) + v(B)$,
	\item \emph{monotone} if $v(A)\leq v(B)$ whenever $A\subseteq B$,
	\item \emph{nonnegative} if $v(A)\geq 0$,
   \end{itemize}
   for all $A,B\in \calL$.   
\end{definition}

\noindent
Let $G(\calL)$ be the set of all games on $\calL$. We consider these subsets of $G(\calL)$:
\begin{align*}
	G_S(\calL) & \coloneqq \{v\in G(\calL) \mid \text{$v$ is supermodular}\},\\
	G_M(\calL) & \coloneqq \{v\in G(\calL) \mid \text{$v$ is  modular}\}.
\end{align*}	
A modular game is also called a \emph{valuation} (over $\dR$) in literature; see
\cite{Birkhoff48,Stanley2012}. Note that $G(\calL)$ is a real vector space
isomorphic to $\dR^{\calL\setminus \{\emptyset\}}$ and therefore $\dim
G(\calL)=\abs{\calL}-1$. One of the bases in $G(\calL)$ is found very
easily. For each nonempty $A\in \calL$, the \emph{unanimity game} $u_A$ is defined by
\begin{equation*}
u_A(B)\coloneqq \begin{cases}
	1 & A\subseteq B,\\
	0 & \text{otherwise,} 
\end{cases}	
\quad \text{for all $B\in \calL$.}
\end{equation*}
Then $\{u_A \mid \emptyset\neq A\in \calL\}$ forms a basis in $G(\calL)$. The coordinates of any game $v\in G(\calL)$ with respect to this basis are calculated using the M\"obius inversion formula \cite{Rota64}. Specifically, the \emph{M\"obius function} of $\calL$ is the function $\mu_{\calL}\colon \calL^2\to \dR$ given recursively as 
\[
\mu_{\calL}(X,Y)\coloneqq
\begin{cases}
1 & X=Y,\\
-\sum\limits_{X\subseteq Z \subset Y} \mu_{\calL}(X,Z) & X\subset Y, \\
0 & \text{otherwise,}
\end{cases}
\qquad \text{for all $X,Y\in \calL$.}
\]
The \emph{M\"obius transform} of $v\in G(\calL)$ is the game $\hat{v} \in G(\calL)$ defined by
\[
\hat{v}(B)\coloneqq\sum_{C\subseteq B} v(C)\cdot  \mu_{\calL}(C,B), \quad B\in\calL.
\]
\begin{lemma}\label{lem:BL}
	For any $v\in G(\calL)$, we have
	\[
	\hat{v}(B)=\sum_{C\subseteq B} v(C)\cdot  (-1)^{\abs{B\setminus C}}, \quad B\in\calL,
	\]
	where the sum above is over all $C\in\calL$ such that $[C,B]$ is a Boolean sublattice of $\calL$.
\end{lemma}
\begin{proof}
	It suffices to apply the observation from \cite[Example 3.9.6]{Stanley2012}. Specifically,  since the lattice $\calL$ is finite and distributive, the formula for M\"obius function $\mu_{\calL}$ simplifies as 
	\[
	\mu_{\calL}(X,Y)=
	\begin{cases}
	(-1)^{\abs{Y\setminus X}} & \text{if $[X,Y]$ is a Boolean lattice,} \\
	0 & \text{otherwise,}
	\end{cases}
	\]
	for every $X,Y\in \calL$ with $X\subseteq Y$.
\end{proof}

Thus, any  $v\in G(\calL)$ can be expressed as a linear combination $v=\sum_{\emptyset \neq A \in\calL} \hat{v}(A)\cdot u_A$, which gives
\begin{equation}\label{eq:coordMoeb}
v(A)=\sum_{B\subseteq A}\hat{v}(B), \quad A\in\calL.
\end{equation}
The set of valuations (modular games) $G_M(\calL)$ is a vector subspace of $G(\calL)$. By Rota's lemma \cite{Rota71}  any valuation on $\calL$ is uniquely determined by its restriction to the set of join-irreducible elements $\calJ$. It follows that the dimension of linear space $G_M(\calL)$ equals $\abs{\calJ}=n$. This means that the polyhedral cone of supermodular games $G_S(\calL)$ is not pointed as it includes the non-trivial linear space $G_M(\calL)$. However, we can always consider the elements of $G_S(\calL)$ modulo $G_M(\calL)$. To this end we introduce the following notion.
\begin{definition}\label{def:0norm}
 A game $v\in G(\calL)$ is said to be \emph{$0$-normalized} if 
 \[
 \hat{v}(A)=0,\quad \text{for all $A\in \calJ$}.
 \]
  Let $G^{\star}(\calL)$ be the set of all $0$-normalized games on $\calL$.
 \end{definition}
  Note that the notion of $0$-normalized game on a distributive lattice $\calL$ coincides with the usual concept of $0$-normalized game in cooperative game theory (see \cite[Definition 2.1.13]{PelegSudholter07}). Indeed, when $\calL$ is the Boolean lattice $2^N$ of all subsets of $N$, then the only join-irreducible elements in $2^N$ are exactly the atoms in $2^N$, that is, $\calJ=\{\{i\} \mid i\in N  \}$. If $v$ is $0$-normalized in sense of Definition \ref{def:0norm}, from \eqref{eq:coordMoeb} we get $v(\{i\})=\hat{v}(\{i\})=0$ for all $i\in N$, which is exactly the definition of $0$-normalized coalitional game on $2^N$.
  
  \begin{lemma}\label{lem:G0}
  	A game $v\in G(\calL)$ is $0$-normalized if and only if $v(A)=v(A^-)$ for all $A\in \calJ$, where $A^-$ is the unique element covered by $A$.
  \end{lemma}
  \begin{proof}
  	Let $A\in\calJ$. Then, for any $B\in \calL$ with $B\subseteq A$, the order interval $[B,A]$ is a Boolean sublattice of $\calL$ if, and only if, either $B=A$ or $B=A^-$. Hence, Lemma \ref{lem:BL} yields $\hat{v}(A)=v(A)-v(A^-)$.
  \end{proof}

   For any $v\in G(\calL)$, put
\begin{equation*}
v^{\star} \coloneqq v- \sum_{B\in \calJ} \hat{v}(B)\cdot u_B.
\end{equation*}
 It is easy to see that  $G^{\star}(\calL)=\{v^{\star} \mid v\in G(\calL)\}$. We claim that, for any $v\in G(\calL)$, there exist uniquely determined $w\in G^{\star}(\calL)$ and $m\in G_M(\calL)$ such that 
 \begin{equation}\label{eq:decomp}
 v=w+m.
 \end{equation}
  Indeed, it suffices to define $w\coloneqq v^{\star}$, $m\coloneqq\sum_{B\in \calJ} \hat{v}(B)\cdot u_B$, and observe that $m\in G_M(\calL)$.
Let $$G^{\star}_S(\calL)\coloneqq G_S(\calL)\cap G^{\star}(\calL).$$ Then  $G^{\star}_S(\calL)=\{v^{\star} \mid v\in G_S(\calL)\}$ and Lemma \ref{lem:G0} says that $G^{\star}_S(\calL)$ contains exactly those supermodular games satisfying $v(A)=v(A^-)$, for all $A\in \calJ$. Moreover, the convex cone $G^{\star}_S(\calL)$ is pointed and polyhedral.

\begin{lemma}\label{lem:monneg}
	Every game $v\in G^{\star}_S(\calL)$ is monotone and nonnegative.
\end{lemma}
\begin{proof}
	Since a monotone game is necessarily nonnegative, it suffices to check monotonicity. We only need to prove that for all $A,B\in\calL$ satisfying $B\subseteq A$ and $\abs{B}=\abs{A}-1$, the inequality $v(B)\leq v(A)$ holds. Since both $A$ and $B$ are down-sets in $(N,\preceq)$, any such $B$ is necessarily of the form $B=A\setminus \{i\}$, where $i$ is a maximal element of $A$ in $(N,\preceq)$. Note that $\da i \subseteq A$ and $B\cap \da i=(A\setminus \{i\})\cap \da i=\DA i$. Then
	supermodularity yields
	\[
	v(A)=v(B\cup \da i) \geq v(B) + v(\da i) - v(\DA i).
	\]
	Since $v$ is $0$-normalized and $\da i\in \calJ$, Lemma \ref{lem:G0} implies $v(\da i) - v(\DA i)=0$.
\end{proof}

By the decomposition \eqref{eq:decomp} we can now write $G_S(\calL)$ as the direct sum of cones, 
\begin{equation}\label{eq:directsum}
G_S(\calL)=G^{\star}_S(\calL) \oplus G_M(\calL).
\end{equation}
Specifically, the identity \eqref{eq:directsum} means that $G_S(\calL)=G^{\star}_S(\calL) + G_M(\calL)$ and $G^{\star}_S(\calL) \cap G_M(\calL)=\{0\}$.
Since $G^{\star}_S(\calL)$ is a pointed polyhedral cone, it is generated by its finitely-many extreme rays.

 In the next section we present a simple linear-algebraic criterion to test if a
 given $0$-normalized supermodular game generates an extreme ray of
 $G^{\star}_S(\calL)$. Our result automatically yields a criterion for
 extremality of games in $G_S(\calL)$: we say  that a supermodular game
   $v\in G_S(\calL)$ is  \emph{extreme} if $v^{\star}$ generates an~extreme ray of $G^{\star}_S(\calL)$. Equivalently, $v\in
 G_S(\calL)$ is extreme if, and only if, the smallest face of $G_S(\calL)$ to
 which $v$ belongs is an atom of the face lattice of $G_S(\calL)$.
Indeed, faces of $G_S(\calL)$ are in one-to-one correspondence with faces of
$G^{\star}_S(\calL)$ by the relation $F=F^{\star}+G_M(\calL)$, where $F$ is a face of
  $G_S(\calL)$ and~$F^{\star}$ a face of $G^*(\calL)$.

\section{The cone of supermodular games}\label{sec:consup}
A \emph{payoff vector} is any vector $x\coloneqq (x_1,\dotsc,x_n)\in\dR^n$ . We define
\[
x(A)\coloneqq \sum_{i\in A} x_i, \quad \text{for any $A\in \calL$,}
\] 
and we always assume $x(\emptyset)\coloneqq 0$. The \emph{core of $v\in G(\calL)$} is a convex polyhedron 
\[
\calC(v)\coloneqq \{x\in\dR^n \mid x(N)=v(N),\; x(A)\geq v(A) \text{ for each $A\in\calL$} \}.
\]
The elements of the core $\calC(v)$ have the standard game-theoretic
interpretation. Namely, no payoff vector $x\in \calC(v)$ can be improved upon by
any coalition $A\in \calL$. In contrast with cores of games over Boolean
lattices, the core of games over distributive lattices can be an unbounded polyhedron. In fact, assume $\calC(v)\neq \emptyset$, where $v\in G(\calL)$. Then $\calC(v)$ is bounded if and only if $\calL$ is a Boolean lattice; see \cite[Chapter 3.3.3]{Grabisch16}. If $v\in G_S(\calL)$, then the polyhedron $\calC(v)$ is pointed and its extreme points $\verts \calC(v)$ are characterized in Theorem~\ref{thm:Cext} below.

 Recall that we always assume that $N$ is partially ordered by~$\preceq$. In addition we also equip $N$ with the total order of natural numbers~$\leq$, so that $(N,\leq)$ becomes a chain.  We say that a permutation $\pi$ of $N$ is \emph{compatible} with $(N,\preceq)$ if $\pi^{-1}$ is an order-preserving map from $(N,\preceq)$ onto $(N,\leq)$. 
   Here, the intended reading is that $i$ is a rank of player $\pi(i)$. Define
\[
\Pi_{\preceq}\coloneqq \left\{\pi \mid \text{$\pi$ is a permutation compatible with $(N,\preceq)$} \right\}.
\]
Compatible permutations are in bijection with maximal chains in $\calL$. Put   $A_0^{\pi}\coloneqq \emptyset$ and $A_i^{\pi}\coloneqq \{\pi(1),\dots,\pi(i)\}$ for each $i\in N$. Then, with each $\pi \in \Pi_{\preceq}$ we associate a maximal chain $\calC^{\pi} \coloneqq \{A_i^{\pi} \mid i\in N \cup \{0\} \}$. Conversely, starting from a maximal chain $\{A_{0},\dots,A_{n}\}$ in $\calL$, where $A_0\subseteq \dots \subseteq A_n$, there is clearly a unique $\pi \in \Pi_{\preceq}$ such that $A_i=A_i^{\pi}$ for each $i\in N \cup \{0\}$.

A \emph{marginal vector} of $v\in G(\calL)$ and $\pi \in \Pi_{\preceq}$ is the  vector $x^{v,\pi} \in \dR^n$ whose coordinates are defined as
\begin{equation}\label{eq:margin}
x^{v,\pi}_{\pi(i)}\coloneqq v(A_i^{\pi}) - v(A_{i-1}^{\pi}), \qquad i \in N.
\end{equation}
It follows directly from the definition of marginal vector that
\begin{equation}\label{eq:keysum}
v(A_i^{\pi})=x^{v,\pi}(A_i^{\pi}),\qquad \text{for all $\pi\in \Pi_{\preceq}$ and all $i \in N\cup \{0\}$}.
\end{equation}
We will make an ample use of the following identity derived from \eqref{eq:keysum}:
\begin{equation}\label{eq:keysum2}
v(A)=x^{v,\pi}(A),\qquad \text{for all $\pi\in \Pi_{\preceq}$ and all $A\in\calC^\pi$}.
\end{equation}
For any $v\in G(\calL)$ and $\pi\in \Pi_{\preceq}$ we define
\[
	\calT^\pi(v) \coloneqq\left\{A\in \calL\mid v(A)=x^{v,\pi}(A) \right\}.
\]
 Each coalition $A \in \calT^\pi(v)$ is said to be \emph{tight} with respect to $v$ and $\pi$. Note that as a consequence of \eqref{eq:keysum2}, the following inclusion holds: 
 \begin{equation}\label{incl:chains}
\calC^\pi\subseteq\calT^\pi(v), \quad \text{for all $v\in G(\calL)$ and all $\pi\in\Pi_\preceq$.}
 \end{equation}

\begin{theorem}\label{thm:Cext}
	Let $\calL$ be a finite distributive lattice and $v\in G(\calL)$. Then the following are equivalent:
\begin{enumerate}
	\item $v\in G_S(\calL)$.
	\item $x^{v,\pi} \in \calC(v)$, for each $\pi\in\Pi_{\preceq}$.
	\item $\verts \calC(v) = \{x^{v,\pi} \mid \pi \in \Pi_{\preceq} \}$.
	\item $v(A)=\min\limits_{\pi\in\Pi_{\preceq}} x^{v,\pi}(A)$, for each $A\in\calL$.
\end{enumerate}	
\end{theorem}

\begin{proof}
	The equivalence of the first three items is well known; see \cite[Theorem 3.27]{Grabisch16}.
We show that 4. implies 2. Let $\sigma \in\Pi_{\preceq}$. Then, for all $A\in \calL$,
\[
x^{v,\sigma}(A) \geq \min\limits_{\pi\in\Pi_{\preceq}} x^{v,\pi}(A) = v(A).
\]
By \eqref{eq:keysum2} we have $x^{v,\sigma}(N)=v(N)$.

From 2. to 4. It suffices to show that for each $A\in\calL$ there exists $\pi\in\Pi_{\preceq}$ such that $v(A)=x^{v,\pi}(A)$. Clearly, we can always find a maximal chain $\calC^\pi$ in $\calL$ such that $A\in\calC^\pi$ for some $\pi\in\Pi_{\preceq}$. Then (\ref{eq:keysum2}) yields $v(A)=x^{v,\pi}(A)$.
\end{proof}	

\begin{remark}
 Many other characterizations of supermodularity can be found in the literature in case that $\calL$ is a Boolean lattice. See \cite[Appendix A]{StudenyKroupa:CoreExtreme} for a comprehensive list of such conditions. In particular, the implication from 2. to 1. was proved by Ichiishi in \cite{Ichiishi81}. The necessary and sufficient conditions involving specific marginal vectors can be found in \cite{vVHN04}.
\end{remark}

Given $v\in G(\calL)$ let $\bx^{v}\colon \Pi_{\preceq} \to \dR^{n}$ be defined by $$\bx^{v}(\pi)\coloneqq x^{v,\pi}, \quad \text{for all $\pi\in\Pi_{\preceq}$.}$$ Further, we consider a mapping $\bx\colon G(\calL)\to (\dR^{n})^{\Pi_{\preceq}}$ such that
\[
\bx(v)\coloneqq \bx^v, \quad \text{for all $v\in G(\calL)$.}
\]
As in \cite{StudenyKroupa:CoreExtreme} we call $\bx$ the \emph{payoff-array transformation}.

\begin{lemma}\label{lem:lininj}
	The payoff-array transformation $\bx$ is linear and injective.
\end{lemma}
\begin{proof}
	Linearity is a direct consequence of the identities $x^{v+w,\pi}=x^{v,\pi}+x^{w,\pi}$ and $x^{\alpha v,\pi}=\alpha x^{v,\pi}$, which are true for every $v,w\in G(\calL)$, all $\alpha\in \dR$ and all $\pi\in\Pi_{\preceq}$. Assume that $v,w\in G(\calL)$ satisfy $\bx^v=\bx^w$ and let $A\in\calL$. Then there exists a permutation $\pi\in\Pi_{\preceq}$  such that $A\in \calC^{\pi}$. It follows from \eqref{eq:keysum2} and from the assumption that
	\[
	v(A)=x^{v,\pi}(A)=x^{w,\pi}(A)=w(A).
	\]
	Hence, $\bx$ is injective.
\end{proof}

We describe the range of payoff-array transformation $\bx$ on the set of $0$-normalized games. For any mapping $\by\colon \Pi_{\preceq}\to \dR^n$ we denote $y^{\pi}\coloneqq\by(\pi)\in \dR^n$, for all $\pi\in \Pi_{\preceq}$.

\begin{lemma}\label{lem:imagex}
	Let $\by\colon \Pi_{\preceq}\to \dR^n$. The following are equivalent:
\begin{enumerate}
	\item There is a unique game $v\in G^{\star}(\calL)$ such that $\by=\bx^v$.
	\item These conditions are satisfied:
	\begin{align}
\label{cond1}\tag{$\dagger$}		y^\pi(A)&=y^\sigma(A) \quad  \text{for all $\pi,\sigma\in \Pi_{\preceq}$ and all  $A\in\calC^\pi\cap \calC^\sigma$,} \\
\label{cond2}\tag{$\dagger\dagger$}		y^\pi_i&=0\quad  \text{for all $\pi\in \Pi_{\preceq}$ and all $i\in N$ such that $\da i\in\calC^\pi $.}
	\end{align}
	
\end{enumerate}
\end{lemma}

\begin{proof}
	Let $\by=\bx^v$ for some $v\in G^{\star}(\calL)$. The equality in \eqref{cond1} is a direct consequence of  \eqref{eq:keysum2} since, for any $\pi,\sigma\in \Pi_{\preceq}$ satisfying  $A\in\calC^\pi\cap \calC^\sigma$, we get
	\[
	y^\pi(A)=x^{v,\pi}(A)=v(A)=x^{v,\sigma}(A)=y^\sigma(A).
	\]
	 Further, let $\pi\in \Pi_{\preceq}$ and $i\in N$ satisfy $\da i\in\calC^\pi $. Put $A\coloneqq\da i$ and observe that $A\in\calJ$ by Proposition \ref{pro:embedding}. This implies that the unique predecessor of $A$ in $\calL$ is $A^-=A\setminus \{i\}$ and $A^-\in \calC^\pi$, by maximality of the chain $\calC^\pi$. We obtain
	\[
	y^\pi_i=x^{v,\pi}_i= x^{v,\pi}(A)-x^{v,\pi}(A^-)=v(A)-v(A^-)=0,
	\]
	where the third equality follows from \eqref{eq:keysum2} and the fourth one  from $0$-normalization of~$v$ (Lemma \ref{lem:G0}).
	
	  Conversely, assume that the conditions \eqref{cond1}--\eqref{cond2} are true. If a game $v\in G^{\star}(\calL)$ satisfying $\by=\bx^v$ exists, then it is unique by injectivity of $\bx$ (Lemma \ref{lem:lininj}). The condition \eqref{cond1} guarantees that it is correct to define the game $v$ as
	  \begin{equation}\label{def:gamefrompi}
	  v(A)\coloneqq y^\pi(A)\quad \text{for all $\pi\in \Pi_{\preceq}$ and all $A\in \calC^\pi$.}
	  \end{equation}
	   By the definition, $\by=\bx^v$.
	  
	   It remains to verify that $v$ is $0$-normalized. Let $A\in\calJ$. By Proposition \ref{pro:embedding} it follows that $A=\da i$ for a unique $i\in N$. There exists some compatible permutation $\pi$ satisfying $A\in \calC^\pi$. Hence, by the definition of $v$ and \eqref{cond2},
	   \[
	   v(A)-v(A^-)=y^\pi(A)-y^\pi(A^-)=y^\pi_i=0.
	   \]
	  This means that $v$ is $0$-normalized and the proof is finished.
\end{proof}

\begin{remark}
	A mapping $\by\colon \Pi_{\preceq}\to \dR^n$, whose special case is the payoff-array transformation $\bx$, can be viewed as a finite collection of possibly repeating points in $\dR^n$ labeled by permutations. This interpretation appears in \cite{DeLoera10}, where a map $\by$ from a finite set into $\dR^n$ is termed a \emph{point configuration}.
\end{remark}

\section{Main result}\label{sec:mainresult}
 Denote \[
 N^\pi(v)\coloneqq \{i\in N\mid x^{v,\pi}_i=0\}.\]
 The main theorem gives a simple criterion how to recognize extreme games among all $0$-normalized supermodular games.

\begin{theorem}\label{thm:main}
	Let $v\in G_S^{\star}(\calL)$ be nonzero. Then the following are equivalent:
\begin{enumerate}
	\item The game $v$ is extreme in $G^{\star}_S(\calL)$.
	\item If $\by\colon \Pi_{\preceq}\to \dR^n$ satisfies the conditions
	\begin{align}
		\label{thm:cond1} \tag{$*$}		y^\pi(A)&=y^\sigma(A) \quad  \text{for all $\pi,\sigma\in \Pi_{\preceq}$ and all  $A\in \calT^\pi(v) \cap \calT^\sigma(v)$,} \\
		\label{thm:cond2}\tag{$**$}		y^\pi_i&=0\quad  \text{for all $\pi\in \Pi_{\preceq}$ and all $i\in N^\pi(v)$,}
	\end{align}	
	then $\by=\alpha\bx^v$, for some $\alpha\in \dR$.
\end{enumerate}
\end{theorem}
\noindent
 We prepare a lemma to be used in the proof of Theorem \ref{thm:main}.  For any $v\in G(\calL)$, put
 \begin{equation}\label{def:Fv}
 \calF_v \coloneqq\bigl\{\{A,B\}\subseteq \calL \mid  v(A\cup B) +v(A\cap B)= v(A)+v(B), \;A || B\bigr\}.
 \end{equation}
 where \[
 A || B \text{ means $A\not\subseteq B$ and $B\not\subseteq A$.}
 \]
 
 \noindent
For any point configuration $\by\colon \Pi_{\preceq}\to \dR^n$ and a game $v\in G(\calL)$, we consider the following property:
 \begin{equation}\label{eq:simplecond}
\begin{split}
y^{\pi}(A)= y^{\sigma}(A), \quad &\text{for each $\{A,B\}\in \calF_v$ and all $\pi,\sigma\in\Pi_\preceq$}\\ & \text{such that  $A\cap B,B,A\cup B\in \calC^\pi$ and $A \in \calC^\sigma$.}
\end{split}
\end{equation}

\begin{lemma}\label{th:1}
Let $v\in G_S^{\star}(\calL)$ and let $\by:\Pi_\preceq\rightarrow\dR^n$ be such that $(*)$ and $(**)$ are satisfied. Then $\by$ fullfills $(\dagger)$,$(\dagger\dagger)$, and \eqref{eq:simplecond}.
\end{lemma}
\begin{proof}
Assume that $\by$ satisfies $(*)$ and $(**)$. It is easy to see that $(\dagger)$ and $(\dagger\dagger)$ are true. In order to prove \eqref{eq:simplecond}, let $\{A,B\}\in \calF_v$,  $\pi,\sigma\in\Pi_\preceq$, and $A\cap B,B,A\cup B\in \calC^\pi$, $A \in \calC^\sigma$.  Since
\[
v(A)=v(A\cup B)+v(A\cap B)-v(B)=x^{v,\pi}(A\cup B)+x^{v,\pi}(A\cap B)-x^{v,\pi}(B)=x^{v,\pi}(A),
\]
we get $A\in \calT^\pi(v)$. Hence, $A\in \calT^\pi(v)\cap \calC^\sigma$ and $(*)$ says that $y^{\pi}(A)= y^{\sigma}(A)$.
 \end{proof}

\begin{proof} (of Theorem \ref{thm:main})
Let $v\in G^\star_S(\calL)$ be nonzero. 
We need to show that $v$ is extreme if and only if the following inclusion holds true: 
\begin{equation}\label{eq:result}
\{\by:\Pi_\preceq\rightarrow\dR^n\mid \by\text{ satisfies }(*),(**)\}\subseteq \{\alpha x^v\mid
\alpha\in\dR\}.
\end{equation}
By the Minkowski-Weyl-Farkas theorem (see \cite[Theorem 3.34]{AliprantisTourky07}), $v$ is extreme if and only if  $v$ belongs to the one-dimensional solution space of some set of tight inequalities for $G^\star_S(\calL)$ of the form
$w(A\cup B) + w(A\cap B) - w(A) -w(B) \geq 0$, for all $A,B\in \calL\setminus\{\emptyset,N\}$. Define 
\[G(v) \coloneqq  \{w\in G^\star(\calL)\mid w(A\cup B)+w(A\cap B)=w(A)+w(B), \text{ for all $\{A,B\}\in\calF_v$}\},\] 
where $\calF_v$ is as in \eqref{def:Fv}.
 Thus, extremality of $v$ is equivalent to the condition
\begin{equation}\label{eq:tempx}
G(v) = \{\alpha v \mid \alpha \in \dR\}.
\end{equation}
Putting $\bx(G(v))\coloneqq \{\bx^w \mid w\in G(v)\}$ and using Lemma  \ref{lem:lininj}, it is immediate that \eqref{eq:tempx} holds if and only if
\begin{equation}\label{eq:eq:p2}
\bx(G(v))=\{\alpha x^v\mid
\alpha\in\dR\}.
\end{equation}
We claim that 
\begin{equation}\label{eq:pp3}
\bx(G(v))\supseteq \{\by:\Pi\rightarrow\dR^n\mid \by\text{ satisfies $(\dagger),(\dagger\dagger)$, and \eqref{eq:simplecond}}\}.
\end{equation}
Let $\by$ satisfies $(\dagger),(\dagger\dagger)$, and \eqref{eq:simplecond}. Lemma \ref{lem:imagex} provides a unique $w\in G^{\star}(\calL)$ such that $\by=\bx^{w}$. We need to verify that $w\in G(v)$. To this end, let $\{A,B\}\in \calF_{v}$. Pick permutations $\pi,\sigma\in \Pi_{\preceq}$ such that $A\cap B,B,A\cup B\in \calC^\pi$ and $A\in \calC^\sigma$. Then \eqref{eq:simplecond} shows that
\[
w(A\cup B)+w(A\cap B)-w(B) =y^\pi(A)=y^\sigma(A)=w(A).
\]
Hence, $w\in G(v)$. Finally, from \eqref{eq:eq:p2}, \eqref{eq:pp3}, and Lemma~\ref{th:1}  we get \eqref{eq:result}, and the proof is finished.
\end{proof}	

We will apply Theorem \ref{thm:main} to the cone of supermodular games on the distributive lattice $\calL$ from Example \ref{example}.  The computations were carried out in the package Convex for Maple \cite{Franz16}.
\begin{example}
 The cone $G_S^{\star}(\calL)$ is embedded into $\dR^9$ and its dimension is $5$.  It has $6$ extreme rays. We will enumerate their minimal integer generators. The parentheses and commas are omitted for the sake of brevity in what follows. Whenever $v_i(A)$ is missing, we put $v_i(A)\coloneqq 0$.
\begin{itemize}
\item $v_1(24)=v_1(234)=v_1(N)=1$.
\item $v_2(34)=v_2(234)=v_2(N)=1$.
\item $v_3(23)=v_3(123)=v_3(234)=v_3(N)=1$.
\item $v_4(234)=v_4(N)=1$.
\item $v_5(23)=v_5(24)=v_5(34)=v_5(123)=1,\; v_5(234)=v_5(N)=2$.
\item $v_6(N)=1$.
\end{itemize}
We will check that $v_1$ is extreme using Theorem \ref{thm:main}. 
Since there are $8$~maximal chains in $\calL$, there are $8$ compatible permutations: $\pi_1=(2 3 1 4), \pi_2=(2 3 4 1), \pi_3=(2 4 3 1), \pi_4=(3 2 4 1), \pi_5=(3 2 4 1), \pi_6 = (3 4 2 1), \pi_7=(4 2 3 1)$, and $\pi_8=(4 3 2 1)$.
 Let $I_1\coloneqq \{1,\dots,5\}$ and $I_2\coloneqq \{6,7,8\}$. There are only $2$~marginal vectors associated with $v_1$,
\[
x^{v_1,\pi_i}=
\begin{cases}
(0,0,0,1) & i \in I_1, \\
(0,1,0,0) & i\in I_2. 
\end{cases}
\]
This means that the tight sets are
\[
\calT^{\pi_i}(v_1)=
\begin{cases}
\{\emptyset,2,3,23,24,123,234,N\} & i\in I_1, \\
\{\emptyset,3,4,24,34,234,N\} & i\in I_2. 
\end{cases}
\]
Hence, the conditions $(*)$ and $(**)$ for $\by\colon \Pi_\preceq \to \dR^4$ are in the form of linear equalities, for all $i\in I_1$ and $j\in I_2$:
\begin{align}
y_1^{\pi_i}=y_2^{\pi_i}=y_3^{\pi_i}=0 \label{lineql1} \\
y_1^{\pi_j}=y_3^{\pi_j}=y_4^{\pi_j}=0 \label{lineql2} \\
y_3^{\pi_i}=y_3^{\pi_j} \\
y_2^{\pi_i} + y_4^{\pi_i} =y_2^{\pi_j} + y_4^{\pi_j} \label{lineql3} \\
y_2^{\pi_i} +y_3^{\pi_i}+ y_4^{\pi_i} =y_2^{\pi_j} +y_3^{\pi_j}+ y_4^{\pi_j} \\
y_1^{\pi_i} +y_2^{\pi_i}+y_3^{\pi_i}+ y_4^{\pi_i} =y_1^{\pi_j}+y_2^{\pi_j} +y_3^{\pi_j}+ y_4^{\pi_j} 
\end{align}
The linear system above has a unique solution up to a real multiple.
Observe that $y_4^{\pi_i}=y_2^{\pi_j}$, for all $i\in I_1$ and $j\in I_2$, as a consequence of \eqref{lineql1}, \eqref{lineql2}, and \eqref{lineql3}. Let $\alpha\in \dR$. Then necessarily $\by=\alpha \bx^{v_1}$. Thus, $v_1$ is extreme by Theorem \ref{thm:main}.
\end{example}

\begin{remark}
It is natural to ask for a game-theoretic meaning of the extreme supermodular games. Since the supermodular cone is finitely-generated, every supermodular game is a conic combination of the extreme ones. There are important solution concepts \cite{PelegSudholter07}, such as the core or Shapley value, which are linear maps on the supermodular cone. Hence, such solution concepts preserve every conic combination of supermodular games. From this viewpoint, extreme supermodular games play the role of basic building block since they fully determine values of any linear solution concept on the supermodular cone.
\end{remark}

\section{Faces and core structure}\label{sec:faces}
Let $\Phi(G_S(\calL))$ be the \emph{face lattice} of $G_S(\calL)$, that is, the family of all nonempty faces of $G_S(\calL)$ ordered by inclusion $\subseteq$. In what follows we will describe the structure of this face lattice. For any subset $G\subseteq G_S(\calL)$ we define
\[
[G] \coloneqq \bigcap \{F\in \Phi(G_S(\calL)) \mid F\supseteq G\},
\]
 the smallest face containing $G$. Join $\vee$ and meet $\wedge$ in $\Phi(G_S(\calL))$  are computed as 
\[
 F\sqcup G  = [F\cup G] \enskip \text{and} \enskip  F\sqcap G  = F \cap G, \qquad \text{for all $F,G\in \Phi(G_S(\calL))$.}
 \]
 For any face $F$, let $\relint F$ be the relative interior of $F$. Put $\Phi'(G_S(\calL))=\{\relint F \mid F \in \Phi(G_S(\calL))\}$. Then $\Phi'(G_S(\calL))$ is a lattice isomorphic to $\Phi(G_S(\calL))$ in which the top is $\relint G_S(\calL)$, the bottom is $\emptyset$, and the join and the meet are given by
  \begin{align*}
 \relint F\vee \relint G & = \relint (F\sqcup G),\\
 \relint F\wedge \relint G & = \relint (F \sqcap G).
 \end{align*}
\noindent
 The following lemma describes the relation between tight sets of $v\in G_S(\calL)$ and faces of $G_S(\calL)$.
 \begin{lemma}\label{lem:eq}
Let $v\in G_S(\calL)$. The following holds.
\begin{enumerate}
\item Let $\pi\in\Pi_{\preceq}$ and $A,B\in\calT^{\pi}(v)$, $A || B$. Then
  $\{A,B\}\in\calF_{v}$.
\item Let $\{A,B\}\in\calF_v$ and $\pi\in\Pi_{\preceq}$ such that $B,A\cup B,
  A\cap B\in \calC^\pi$. Then $A\in \calT^\pi(v)\setminus \calC^\pi$.
\end{enumerate}
\end{lemma}

 \begin{proof}
   1. Let $v\in G_S(\calL)$, $\pi \in\Pi_{\preceq}$, and $A,B\in
     \calT^{\pi}(v)$ with $A || B$. It is well known that
     $\calT^{\pi}(v)$ is closed under union and intersection for supermodular
     games. Hence, the equality 
 $x^{v,\pi}(A) + x^{v,\pi}(B)=x^{v,\pi}(A\cup B) + x^{v,\pi}(A\cap B)$ yields
$
v(A) + v( B)  = v(A\cup B) + v(A\cap B)
$.
 Therefore, $\{A,B\}\in \calF_{v}$.

2. Let $\{A,B\}\in\calF_v$ and let $\pi\in\Pi_\preceq$ be such that $B,A\cup B,A\cap B\in \calC^\pi$. Then
\[
v(A) = v(A\cup B)+v(A\cap B) - v(B) = x^{v,\pi}(A\cup B) + x^{v,\pi}(A\cap B)-x^{v,\pi}(B) = x^{v,\pi}(A).
\]
Since $A || B$, we get  $A\in\calT^\pi(v)\setminus \calC^\pi$. 
 \end{proof}

 Games in $G_S(\calL)$ belong to the same face if and only if they possess identical structure of their tight sets. Precisely:
 \begin{proposition}\label{pro:relint}
 Let $v,w\in G_S(\calL)$. The following are equivalent.
 \begin{enumerate}
 \item $v,w \in \relint F$, for some $F\in \Phi(G_S(\calL))$.
 \item $\calT^{\pi}(v)=\calT^{\pi}(w)$, for every $\pi\in\Pi_\preceq$.
 \end{enumerate}
 \end{proposition}
 \begin{proof}
 Let $v,w \in \relint F$ for some $F\in \Phi(G_S(\calL))$. If they are linearly dependent, then the statement is trivial. Assume that $v$ and $w$ are linearly independent and let $L$ be the unique line in the linear space $G(\calL)$ such that $v,w\in L$. Then $L\cap (F\setminus \relint F)=\{u,u'\}$, where $u\neq u'$. Since $v,w \in \relint F$ there exist $0 <\alpha,\beta < 1$ such that $v=\alpha u + (1-\alpha) u'$ and $w=\beta u + (1-\beta) u'$.
 
 Let $\pi\in\Pi_\preceq$ and $A\in \calT^{\pi}(v)$. By linearity (Lemma \ref{lem:lininj}) we get
 \[
 \alpha x^{u,\pi}(A) + (1-\alpha) x^{u',\pi}(A) =x^{v,\pi}(A)=v(A)=\alpha u(A) + (1-\alpha) u'(A).
 \]
Since $x^{u,\pi}\in \calC(u)$ and  $x^{u',\pi}\in \calC(u')$, this implies  $x^{u,\pi}(A)=u(A)$ and $x^{u',\pi}(A)=u'(A)$, which means that $A\in \calT^{\pi}(u)\cap \calT^{\pi}(u')$. Hence,
 \[
 w(A)=\beta u(A) + (1-\beta) u'(A) = \beta x^{u,\pi}(A) + (1-\beta) x^{u',\pi}(A) =x^{w,\pi}(A).
 \]
 This proves the inclusion $\calT^{\pi}(v)\subseteq \calT^{\pi}(w)$. The opposite inclusion is established analogously.
 
 To prove the converse, assume $\calT^{\pi}(v)=\calT^{\pi}(w)$ for all $\pi\in\Pi_\preceq$. It suffices to show that $\calF_{v}=\calF_{w}$, where $\calF_{v}$ is as in \eqref{def:Fv}, since this already implies existence of a unique $F\in \Phi(G_S(\calL))$ such that $v,w \in \relint F$. First, we prove  
 \begin{equation}\label{incl1}
 \calF_{v}\subseteq \calF_{w}.
\end{equation}
Let $\{A,B\}\in \calF_{v}$. There exists $\pi\in \Pi_\preceq$ such that $A\cap
B, A, A\cup B\in \calC^{\pi}$. Hence,  $A\in\calT^\pi(v)$ and 
\begin{align*}
v(B) & = v(A\cup B) + v(A\cap B) - v(A) \\
& =x^{v,\pi}(A\cup B) + x^{v,\pi}(A\cap B) - x^{v,\pi}(A) = x^{v,\pi}(B),
\end{align*}
which means $B\in \calT^{\pi}(v)=\calT^{\pi}(w)$. Then $A,B\in
\calT^{\pi}(w)$, and by Lemma \ref{lem:eq} (1) $\{A,B\}\in \calF_{w}$, so \eqref{incl1} holds. The proof of inclusion $\calF_{w}\subseteq \calF_{v}$ is analogous.
 \end{proof}

 Let  $\calS(\calL)$ be the lattice of all sublattices of $\calL$ ordered by set inclusion~$\subseteq$. The \emph{core structure of $v\in G_S(\calL)$}  (cf. \cite{KuipersVermeulenVoorneveld10} and \cite[Definition 4]{StudenyKroupa:CoreExtreme})  is the mapping $\calT(v)\colon \Pi_{\preceq}\to \calS(\calL)$ defined as $$\pi \in \Pi_{\preceq}\; \mapsto\; \calT^{\pi}(v).$$
 \noindent
 The above definition is correct since $\calT^{\pi}(v)$ is a lattice as a consequence of Lemma \ref{lem:eq}.
 By Proposition \ref{pro:relint}, $\calT(v)=\calT(w)$ for all $v,w \in \relint F$. Hence, we may define a mapping \[
 \calT\colon \Phi(G_S(\calL)) \to \calS(\calL)^{\Pi_{\preceq}}
 \]
 by
 \[
 \calT(F) \coloneqq \calT(v), \quad \text{for any $v\in \relint F$.}
 \]
 We will order the elements of $\calS(\calL)^{\Pi_{\preceq}}$ by the product order $\subseteq$ inherited from $\calS(\calL)$. Specifically, for any $\calU,\calV \in \calS(\calL)^{\Pi_{\preceq}}$,
 \[
 \calU \subseteq  \calV \text{ whenever $\calU^{\pi} \subseteq  \calV^{\pi}$},\qquad \text{for all $\pi \in \Pi_{\preceq}$}.
 \]

  \begin{proposition} 
  The mapping $\calT$ is injective, order-reversing, and its inverse $\calT^{-1}$ is also order-reversing.
 \end{proposition}
 \begin{proof}
$\calT$ injective is an easy consequence of Proposition \ref{pro:relint}. 
We will prove that $\calT$ is an order-reversing map.
Let $F_{1}\subseteq F_{2}$ be faces of $G_S(\calL)$ and select arbitrarily $v_{1}\in \relint F_{1}$ and $v_{2}\in \relint F_{2}$.  We want to show 
\begin{equation}\label{incl2}
\calT^{\pi}(v_{2})\subseteq \calT^{\pi}(v_{1}) \qquad \text{for every $\pi\in \Pi_{\preceq}$.} 
\end{equation}
Let $\pi\in \Pi_{\preceq}$ and $A\in \calT^{\pi}(v_{2})$. Using
(\ref{eq:margin}), this is equivalent to saying that $v_2$ is a solution of the
equation in $v$:
\begin{equation}\label{eq:sol}
\sum_{i\in A} (v(A^\pi_{\pi^{-1}(i)}) - v (A^\pi_{\pi^{-1}(i-1)})) = v(A).
\end{equation}
As this equation is satisfied by all games in $\relint F_2$ and only these ones,
it follows that (\ref{eq:sol}) is implied by the equalities determining $\relint
F_2$, that is, those corresponding to $\calF_{v_2}$. As $F_1\subseteq F_2$,
$\relint F_1$ is determined by a~superset of equalities, and therefore the
equality (\ref{eq:sol}) is also satisfied by $v_1$.  Hence $A\in
\calT^\pi(v_1)$.

To show that $\calT^{-1}$ is order-reversing, let $\calT^{\pi}(v_{1})\subseteq
\calT^{\pi}(v_{2})$ for all $\pi\in\Pi_{\preceq}$, where $v_{1}\in\relint F_{1}$
and $v_{2}\in\relint F_{2}$, for some faces $F_{1}$ and $F_{2}$. We will prove
that $F_{2}\subseteq F_{1}$, which is the same as $\calF_{v_{1}}\subseteq
\calF_{v_{2}}$. Let $\{A,B\}\in\calF_{v_1}$. Then there exists $\pi$ s.t. $A\cap
B,A,A\cup B\in \calC^\pi$. Hence, by Lemma~\ref{lem:eq} (2), $B\in
\calT^\pi(v_1)\setminus\calC^\pi$, hence $B\in\calT^\pi(v_2)\setminus
\calC^\pi$. As $A\in\calC^\pi\subseteq \calT^\pi(v_2)$ and $A\| B$, by
Lemma~\ref{lem:eq} (1), $\{A,B\}\in \calF_{v_2}$. 
 \end{proof}

\begin{corollary}
$\calT$ is a lattice isomorphism from the face lattice $\Phi(G_S(\calL))$ onto a sublattice of $\calS(\calL)^{\Pi_{\preceq}}$.
\end{corollary}

\begin{remark}
The same reasoning can be applied to the face lattice of all $0$-normalized supermodular games, $\Phi(G_{S}^{\star}(\calL))$. Indeed, it follows from the direct sum decomposition \eqref{eq:directsum} that $\Phi(G_{S}^{\star}(\calL))$ and $\Phi(G_{S}(\calL))$ are isomorphic lattices.
\end{remark}

\section{Facets of the cone of supermodular games}\label{sec:facets}
Kuipers et al. \cite[Corollary 11]{KuipersVermeulenVoorneveld10} characterised those supermodular inequalities which determine the facets of $G_{S}(\calL)$ when $\calL=2^{N}$. Specifically, the facet-determining inequalities are 
\begin{equation}\label{ineq:Kuipers}
v(A\cup \{i,j\}) - v(A\cup\{i\}) - v(A\cup\{j\}) + v(A)\geq 0
\end{equation}
for all $A\subseteq N$ and every pair of distinct $i,j\in N\setminus A$. The next theorem identifies the facets of the cone $G_{S}(\calL)$, for any $\calL$. Since $\calL$ is the lattice of down-sets of a poset $(N,\preceq)$, if $A\in\calL$ and $i\in N$ are such that $\DA i\subseteq A$,  then $A\cup \{i\}\in \calL$.
\begin{theorem}\label{thm:facets}
The facets of $G_S(\calL)$ are given by the inequalities of the form
\begin{equation}\label{eq:p11}
v(A\cup \{i,j\}) - v(A\cup\{i\}) - v(A\cup\{j\}) + v(A)\geq 0
\end{equation}
with $A\in \calL$ and distinct $i,j\in N\setminus A$ such that $\DA
i\subseteq A$ and $\DA j\subseteq A$.
\end{theorem}
\begin{proof}
Observe that, since the lattice $\calL$ is ranked\footnote{A \emph{ranked} (also \emph{graded}) \emph{lattice} is a lattice in which all the maximal chains have the same cardinality.}, if $A$ and $A\cup B$ with
$B\subseteq N\setminus A$ are in $\calL$, then $A\cup\{i\}\in\calL$ for some
$i\in B$. 

First, we show that any  supermodular inequality can be derived from those of type
(\ref{eq:p11}). Consider
\begin{equation}\label{eq:p2}
v(A\cup B)-v(A) -v(B) +v(A\cap B)\geq 0, \quad (A,B\in\calL)
\end{equation} 
with $|A\Delta B|>2$. We show by induction on $|A\Delta B|$ that (\ref{eq:p2})
can be derived from (\ref{eq:p11}). First we establish the result for $|A\Delta
B|=3$. Take
\begin{equation}\label{eq:p3}
v(A\cup\{i,j,k\}) - v(A\cup\{i,j\}) - v(A\cup\{k\})+v(A)\geq 0,
\end{equation} 
with $A\cup\{i,j\}),A\cup\{k\}\in\calL$. Supposing $A'\coloneqq A\cup \{i\}\in\calL$ (by
the preliminary remark, either $A\cup\{i\}$ or $A\cup\{j\}$ belongs to $\calL$),
we have by the assumption
\begin{align*}
v(A'\cup \{j,k\})-v(A'\cup\{j\})-v(A'\cup\{k\})+v(A')& \geq 0,\\
v(A\cup \{i,k\})-v(A\cup\{i\})-v(A\cup\{k\})+v(A)& \geq 0,
\end{align*}
whose sum yields (\ref{eq:p3}). 

We suppose that the result holds for $|A\Delta B|=k$, where $k\geq 3$,
and let us prove it for $|A\Delta B|=k+1$. Consider the inequality
\begin{equation}\label{eq:p4}
v(A\cup\{i_1,\ldots,i_\ell,\ldots,i_{k+1}\})  - v(A\cup\{i_1,\ldots,i_\ell\})  -
v(A\cup\{i_{\ell+1},\ldots, i_{k+1}\})+v(A)\geq 0,
\end{equation}
with $A\cup\{i_1,\ldots,i_\ell\},A\cup\{i_{\ell+1},\ldots,
i_{k+1}\}\in\calL$. Put $A'\coloneqq A\cup\{i_{\ell+1}\}\in\calL$. It follows from the assumption that
\begin{align*}
v(A'\cup\{i_1,\ldots,i_\ell,i_{\ell+2},\ldots,i_{k+1}\}) - v(A'\cup
  \{i_1,\ldots,i_\ell\}) - v(A'\cup\{i_{\ell+2},\ldots, i_{k+1}\}) + v(A') &
  \geq 0\\
v(A\cup\{i_1,\ldots,i_{\ell+1}\}) - v(A\cup \{i_{\ell+1}\}) -
v(A\cup\{i_1,\ldots,i_\ell\}) + v(A)& \geq 0,
\end{align*}
whose sum gives (\ref{eq:p4}). 

Second, we prove that no inequality of type (\ref{eq:p11}) is redundant.  It is clearly sufficient to prove the result for the Boolean lattice $\calL=2^N$. Consider (\ref{eq:p11}) for fixed $A,i,j$ and define the following game $v_{A,i,j}$ for an arbitrary $\epsilon>0$:
\[
v_{A,i,j}(B) \coloneqq \begin{cases}
  \epsilon & \text{ if } B=A\cup\{i\},\\
  0 & \text{ if } B\subset A\cup\{i\} \text{ or } B=A\cup\{j\}\text{ or } B\supset
  A\cup\{i\},\\
  -\epsilon & \text{ otherwise.}
  \end{cases}
\]
We claim that $v_{A,i,j}$ satisfies all the 
inequalities (\ref{eq:p11}) except the one with $A, i,j$. This proves that the
latter inequality is not redundant. 

{\it Proof of the claim.} Put $v\coloneqq v_{A,i,j}$, $i\coloneqq 1,j \coloneqq 2$. Clearly,
$$v(A\cup\{1,2\})-v(A\cup\{1\})-v(A\cup\{2\})+v(A)=-\epsilon<0.$$ Consider the
quantity \[\Delta(B,i,j)\coloneqq v(B\cup\{i,j\})-v(B\cup\{i\})-v(B\cup\{j\})+v(B).\] We
first study $\Delta(B,i,j)$ when one of the terms has value $\epsilon$. If
$v(B\cup\{i,j\})=\epsilon$ or $v(B)=\epsilon$, the other terms can be 0 or $-\epsilon$,
so that $\Delta(B,i,j)\geq 0$. If $v(B\cup \{i\})=\epsilon$ (i.e., $B=A$, $i=1$), then $v(B\cup
\{j\})=-\epsilon$, unless $j=2$. Supposing this is not the
case, we have then $v(B\cup\{i,j\})=v(B)=0$. Hence, $\Delta(B,i,j)=0$.

We study now $\Delta(B,i,j)$ when the values of the terms are either 0 or
$-\epsilon$. Suppose that $v(B\cup\{i\})=v(B\cup\{j\})=0$. Both 
$B\cup\{i\}$ and $B\cup\{j\}$ are either proper subsets of $A\cup\{1\}$ or
proper supersets of it. Then both $B$ and $B\cup\{i,j\}$ are either subsets or supersets of $A\cup\{1\}$, and the equality is impossible since we excluded the value $\epsilon$
for $v$. Therefore, $v(B\cup\{i,j\})=v(B)=0$, and $\Delta(B,i,j)=0$. Suppose now
that $v(B\cup\{i\})=0$ and $v(B\cup\{j\})=-\epsilon$. Then either
$B\cup\{i\}\subset A\cup\{1\}$ or $B\cup\{i\}\supset A\cup\{1\}$, and in both
cases $j\not\in A\cup\{1\}$. It follows that, if $B\cup\{i\}\subset
A\cup\{1\}$, then $B\subset A\cup\{1\}$, therefore $v(B)=0$, yielding
$\Delta(B,i,j)\geq 0$. If $B\cup\{i\}\supset A\cup\{1\}$, then
$v(B\cup\{i,j\})=0$, hence the same conclusion holds.
Finally, suppose $v(B\cup\{i\})=v(B\cup\{j\})=-\epsilon$. Then, in any case, 
$\Delta(B,i,j)\geq 0$.
\end{proof}

 \begin{example}
 Let $(N,\preceq)$ and $\calL$ be as in Example \ref{example}. Applying Theorem \ref{thm:facets} we see that the facets of $G_S(\calL)$ are in bijection with the linear inequalities of the form 
 \begin{enumerate}
 \item $v(\{i,j\})\geq v(\{i\})+v(\{j\})$, for all distinct $i,j\in N\setminus \{1\}$,
 \item $v(A\cup B) + v(A\cap B)\geq v(A)+v(B)$, for all distinct $A,B \in \calL$ satisfying $|A|=|B|=2$,
 \item $v(N)+v(\{2,3\}) \geq v(\{1,2,3\})+v(\{2,3,4\})$.
 \end{enumerate}
 Thus, there are $7$ facets, whereas the cone $G_S(2^N)$ of supermodular games on the Boolean lattice $2^N$ has $\binom{n}{2}\cdot 2^{n-2}=24$ facets by \eqref{ineq:Kuipers}.
 \end{example}

\section*{Acknowledgements}
The work of Tom\'{a}\v{s} Kroupa has been supported from the Czech Science Foundation project GA16-12010S.

\end{document}